\documentclass[11pt]{amsart}
\usepackage{amssymb,amsmath,amsthm,verbatim,mathrsfs}

\newtheorem{thm}{Theorem}

\newtheorem{defn}[thm]{Definition}
\newtheorem{lem}[thm]{Lemma}
\newtheorem{cor}[thm]{Corollary}

\theoremstyle{remark}
\newtheorem*{rem}{Remark}

\newcommand{\del}{\partial}

\newcommand{\ddt}{\frac{\del}{\del t}}
\renewcommand{\epsilon}{\varepsilon}

\title{On the positive mass theorem for manifolds with corners}

\author{Donovan McFeron and G\'abor Sz\'ekelyhidi}
\address{School of Theoretical and Applied Science, Ramapo College of
New Jersey, Mahwah, NJ}
\email{dmcferon@ramapo.edu}
\address{Department of Mathematics, University of Notre Dame, Notre
Dame, IN}
\email{gszekely@nd.edu}

\begin{document}

\begin{abstract}
We study the positive mass theorem for certain non-smooth
metrics following P. Miao's work. Our approach is to smooth the metric
using the Ricci flow. As well as improving some previous results on the
behaviour of the ADM mass under the Ricci flow, we extend the analysis of the
zero mass case to higher dimensions.
\end{abstract}

\maketitle
\section{Introduction}
The positive mass theorem, roughly speaking, says that an asymptotically flat
manifold with non-negative scalar curvature has non-negative ADM mass.
This was proved by Schoen-Yau~\cite{SY79} for manifolds of dimension at
most 7, and by Witten~\cite{Wit81} (see also Bartnik~\cite{Bar86}) for spin
manifolds of arbitrary dimension. The ADM mass is an invariant
of the manifold introduced by Arnowitt-Deser-Misner~\cite{ADM61}, 
depending on the geometry at infinity. For precise
definitions see Section~\ref{sec:prelim}.

It is a natural question to generalise the positive mass theorem to
non-smooth Riemannian metrics. Part of the problem is to define what
non-negative scalar curvature means when the metric is not sufficiently
regular for the scalar curvature to be defined in the usual way. 

We study this problem for a special type of singular metric, following
P. Miao~\cite{Miao02}. He considers a smooth manifold $M$ with a compact
domain $\Omega\subset M$ such that $M\setminus\Omega$ is diffeomorphic
to $\mathbf{R}^n$ minus a ball, and $\Sigma=\partial\Omega$ is a smooth
hypersurface. The metric $g$ on $M$ is assumed to be Lipschitz, such
that its restriction to $\Omega$ and $M\setminus\overline{\Omega}$ is
$C^{2,\alpha}$. Away from $\Sigma$ the scalar curvature of $g$ can be defined
as usual. On $\Sigma$ write $H(\Sigma_-)$ for the mean curvature of
$\Sigma$ in $(\overline{\Omega}, g)$, and $H(\Sigma_+)$ for the mean
curvature of $\Sigma$ in $(M\setminus\Omega, g)$, in both cases with
respect to the normal vectors pointing out of $\Omega$. As explained
in~\cite{Miao02}, the condition
\begin{equation} \label{eq:H}
	H(\Sigma_-) \geqslant H(\Sigma_+)
\end{equation}
can be thought of as non-negativity of the scalar curvature along
$\Sigma$ in a distributional sense. The main theorem in~\cite{Miao02} is
\begin{thm}[Miao~\cite{Miao02}]\label{thm:Miao}
	Suppose that $M$ is a manifold for which the positive mass
	theorem of Schoen-Yau~\cite{SY79} and Witten~\cite{Wit81} holds
	for smooth asymptotically flat metrics. Suppose now
	that $g$ is a
	metric on $M$ with corners along a hypersurface $\Sigma$ as described
	above, and that $g$ is asymptotically flat in
	$C^2_{\delta}$ where $\delta > (n-2)/2$. 
	In addition, suppose that 
	the scalar curvature of $g$ is non-negative and
	integrable away from $\Sigma$ and in addition the condition
	(\ref{eq:H}) holds. Then the ADM mass $m(g)\geqslant 0$. 
\end{thm}

In this paper we will first give a slightly different proof of this result,
although some ingredients will be used from~\cite{Miao02}. The basic
idea in~\cite{Miao02}
is to first obtain smoothings $g_\epsilon$ of the metric $g$ using
mollifiers, and then show
that for small $\epsilon$ one can conformally change $g_\epsilon$ to
make the scalar curvature non-negative. One then applies the usual
positive mass theorem. Instead we use the Ricci flow introduced by
Hamilton~\cite{Ham82} to
smooth the metric, which at least heuristically should preserve the
non-negative scalar curvature condition. It was shown by
Simon~\cite{Sim02} that the Ricci flow can be started with a $C^0$
initial metric, and the construction relies on smoothings $g_\epsilon$
as above. We then use the bounds on the scalar curvature of $g_\epsilon$
obtained by Miao to show that the metrics along the Ricci flow starting
at the singular metric $g$ have non-negative scalar curvature. 

When the mass is zero, the corresponding rigidity result 
was shown by Miao for dimension
3 using the results of Bray-Finster~\cite{BF02}. 
We can extend this to any $M$ for which the smooth positive mass
theorem holds. 
\begin{thm}\label{thm:zeromass}
	Under the assumptions of Theorem~\ref{thm:Miao} suppose that
	$m(g)=0$. Then there exists a $C^{1,\alpha}$ diffeomorphism
	$\phi:M\to\mathbf{R}^n$, such that $g=\phi^*g_{Eucl}$, where
	$g_{Eucl}$ is the standard Euclidean metric. 
\end{thm}

\noindent 
The result shows that when the mass is zero, then the metric $g$ is the flat
metric, written in terms of a possibly non-smooth coordinate system.
The zero mass case for Lipschitz metrics was studied also by
Shi-Tam~\cite{ST02} and Eichmair-Miao-Wang~\cite{EMW10} in higher dimensions,
but there the metrics $g$ dealt with are assumed to satisfy equality in
(\ref{eq:H}). It seems likely that their methods extend to the case
where strict inequality holds in (\ref{eq:H}) at least on spin
manifolds. Note that Theorem 2 answers a question of
Bray-Lee~\cite{BL09}, showing that one can remove the spin assumption in
the equality case of their Theorem 1.4. 

The main advantage of using the Ricci flow in proving
Theorem~\ref{thm:zeromass} as opposed to the method of
Miao~\cite{Miao02}, is that when we smooth out the metric using the
Ricci flow, the mass can not increase (with more regularity
assumptions as in Theorem~\ref{thm:mass-flow} below, 
the mass is constant, but when the initial metric is singular we could
only show that the mass does not increase). In contrast in Miao's
method, when we apply a conformal change to the smoothings $g_\epsilon$
to make the scalar curvature non-negative, the mass will only change
slightly,  but it can increase. This is enough to prove non-negativity
of the mass of the singular metric, but it makes analyzing the zero mass
case harder. 

The behaviour of the ADM mass under the Ricci flow has already been
studied by Dai-Ma~\cite{DM07} and also briefly by
Oliynyk-Woolgar~\cite{OW08} in their study of the Ricci flow on
asymptotically flat manifolds. In these works, however, less than
optimal decay conditions are assumed for the metric. We therefore show
the following.  
\begin{thm}\label{thm:mass-flow}
	If the metric $g$ is asymptotically flat in $C^2_{\delta}$ for
	$\delta > (n-2)/2$, and the scalar curvature $R(g)$ is integrable, 
	then
	$R(g(t))\in L^1$ for $t>0$ and the ADM mass is constant along
	the Ricci flow. The short time solution of the Ricci flow used
	here is constructed by Shi~\cite{Shi89}.
\end{thm}

In Section~\ref{sec:prelim} we will recall the basic definitions of
asymptotically flat manifolds and the ADM mass, as well as Simon's
construction~\cite{Sim02} of the Ricci flow with $C^0$ initial data. In
Section~\ref{sec:mass} we prove Theorem~\ref{thm:mass-flow} as well as
some other results that we need later. In Section~\ref{sec:limit} we
study how the mass behaves under taking a limit of a sequence of
metrics. Finally, in
Section~\ref{sec:main} we give the proof of Theorem~\ref{thm:zeromass}.
The proof of a technical lemma is given in the Appendix.

\subsection*{Acknowledgements}
We would like to thank Mu-Tao Wang for helpful discussions, and Duong
Phong for his interest in this work. We would also like to thank Dan Lee
for useful comments on a previous version of this paper.
The second named author is
partially supported by NSF grant DMS-0904223.

\section{Preliminaries}\label{sec:prelim}
\subsection{The ADM mass}
Suppose that $M$ is a non-compact manifold such that there exists a
compact subset $K\subset M$ and functions $x^1,\ldots,x^n:M\setminus
K\to\mathbf{R}^n$, which give a diffeomorphism
\[ (x^1,\ldots,x^n):M\setminus K\to \mathbf{R}^n\setminus B\]
for some ball $B\in\mathbf{R}^n$, which for simplicity we can take to be
the unit ball (after rescaling the metric if necessary). 
For $\delta\in\mathbf{R}$ we define the weighted space
$C^{k,\alpha}_\delta$ as follows (see Bartnik~\cite{Bar86}). Fix a
metric $h$ on $M$ which is the Euclidean metric outside $K$, and let
$\rho:\mathbf{R}^n\to\mathbf{R}$ be a smooth function
equal to $|x|$ outside the ball of
radius 2, and equal to $1$ on the ball of radius 1. Through the
asymptotic coordinates $x^i$ we can think of $\rho$ as a function on
$M$. Then the weighted H\"older norms are defined by
\[ \begin{aligned}
	\Vert f\Vert_{C^{k,\alpha}_\delta} = &\sup_M \rho^{\delta}|f| +
	\sup_M
	\rho^{\delta+1}|\nabla f| +\ldots + \sup_M \rho^{\delta+k}
	|\nabla^k f|\\
 &+ \sup_{x,y\in M} \min\{\rho(x),\rho(y)\}^{
\delta+k+\alpha}\frac{|\nabla^k f(x) - \nabla^k f(y)|}{d(x,y)^\alpha},
\end{aligned}\]
where the derivatives and norms are taken with respect to $h$. A
Riemannian metric $g$ on $M$ is
defined to be asymptotically flat in $C^{k,\alpha}_\delta$ if $\delta >
0$ and we have
\begin{equation}
		g_{ij}-h_{ij} \in C^{k,\alpha}_\delta(M\setminus K).
\end{equation}
In particular if $g$ is asymptotically flat in $C^{1,\alpha}_\delta$,
then there is some constant $C$ such that on $M\setminus K$ we have
\begin{equation}\label{eq:decay1}
	\begin{aligned}
	|g_{ij} - h_{ij}| &< C|x|^{-\delta} \\
	|\nabla g_{ij}| &< C|x|^{-\delta-1},
\end{aligned}
\end{equation}
in terms of the asymptotic coordinates.

If $g$ is a smooth asymptotically flat metric in
$C^{1,\alpha}_\delta$ for $\delta > (n-2)/2$, and the scalar
curvature $R(g)\in L^1$, then the ADM mass (see
Arnowitt-Deser-Misner~\cite{ADM61} and Bartnik~\cite{Bar86}) is defined
by
\[ m(g) = \lim_{r\to\infty} \int_{\partial B_r}
g_{ij,j}-g_{jj,i}\,dS^i,\]
where the derivatives are taken with respect to the asymptotic
coordinates $x^i$. 
The positive mass theorem is then the following. 

\begin{thm}[Schoen-Yau~\cite{SY79}, Witten~\cite{Wit81}] Suppose that
	either $\dim M\leqslant 7$ or that $M$ is a spin manifold. Let
	$g$ be a smooth asymptotically flat metric of order $\delta >
	(n-2)/2$ in $C^{1,\alpha}_\delta$,
	and suppose that the scalar curvature of $g$ is
	integrable, and non-negative. Then $m(g)\geqslant 0$.
	Furthermore, if $m(g)=0$ then $M\cong \mathbf{R}^n$ and $g$ is
	flat. 
\end{thm}

When studying the mass, the following formulas will be useful (see
Bartnik~\cite{Bar86}).
For any metric $g$ we have
\begin{equation}\label{eq:R}
	\begin{aligned}
		R(g) = &|g|^{-1/2} \partial_i\left(|g|^{1/2} g^{ij}
		\left(\Gamma_j
		- \frac{1}{2}\partial_j(\log |g|)\right)\right) \\
		&- \frac{1}{2} g^{ij}\Gamma_i\partial_j(\log |g|) +
		g^{ij}g^{kl}g^{pq}\Gamma_{ikp}\Gamma_{jql},
	\end{aligned}
\end{equation}
where $|g|$ is the determinant, $\Gamma^i_{kl}$ are the Christoffel
symbols and $\Gamma^i=g^{kl}\Gamma^i_{kl}$. For an asymptotically flat
metric in $C^{1,\alpha}_\delta$ for $\delta > (n-2)/2$, we have
\begin{equation}\label{eq:expand}
	|g|^{1/2}g^{ij}\left(\Gamma_j-\frac{1}{2}\partial_j(\log
	|g|)\right) 
	= g_{ij,j} - g_{jj,i} + O\big(|x|^{-1-2\delta}\big),
\end{equation}
where the constant in the error term only depends on $C$. Using this and
integrating \eqref{eq:R} by parts on the region $M\setminus B_r$, we get
\begin{equation}\label{eq:mparts}
	\begin{aligned}
		m(g) =& \int_{M\setminus B_r} R(g)\,dV +
		\int_{\partial B_r} 
		|g|^{1/2}g^{ij}\left(\Gamma_j-\frac{1}{2}\partial_j(\log
		|g|)\right)\,dS^i \\
		&- \int_{M\setminus B_r}\frac{1}{2}g^{ij}\Gamma_i
		\partial_j(\log |g|) + g^{ij}g^{kl}g^{pq} \Gamma_{ikp}
		\Gamma_{jql}\, dV \\
		=& \int_{M\setminus B_r} R(g)\, dV + \int_{\partial B_r}
		g_{ij,j} - g_{jj,i}\,dS^i + O(r^{-\lambda})
	\end{aligned}
\end{equation}
for some $\lambda > 0$. The constant in the error term
$O\left(r^{-\lambda}\right)$ only depends on the $C^{1,\alpha}_\delta$ norm of
$g-h$.

\subsection{Ricci flow on asymptotically flat manifolds}
The Ricci flow is the evolution equation
\begin{equation}\label{eq:Ricflow}
	\ddt g_{ij} = -2R_{ij}
\end{equation}
introduced by Hamilton~\cite{Ham82}, where $R_{ij}$ is the Ricci tensor
of the time-dependent metric $g_{ij}$. It was studied recently
on asymptotically flat manifolds by
Oliynyk-Woolgar~\cite{OW08} and Dai-Ma~\cite{DM07}. Using the results of
Shi~\cite{Shi89} it is clear that a solution of the Ricci flow exists for a
short time with an asymptotically flat initial metric, 
but the question is whether it remains asymptotically flat for positive
time. The result that we need can be stated as follows.  
\begin{thm}\label{thm:AFRF}
	Suppose that $g$ is an asymptotically flat metric of order
	$\delta > 0$ in $C^2_{\delta}$. 
	Suppose that $g(t)$ solves the Ricci flow for
	$t\in[0,T]$, with $g(0)=g$. Then $g(t)$ is asymptotically flat
	in $C^2_\delta$
	of the same order $\delta$. More precisely there is a constant
	$C$ depending on $g$ and $T$, such that 
	\[ \Vert g(t) - h\Vert_{C^2_\delta} < C\]
	for all $t\in[0,T]$. 
\end{thm}
This result could either be proved by systematically working in weighted
spaces as was done in \cite{OW08}, or by using the maximum principle as
in \cite{DM07}. We will use the maximum principle argument to prove a 
slightly different version of this result
in Lemma~\ref{lem:af}. 

Suppose now that the metric $g$ is only $C^0$ on a compact set
$K\subset M$. Using the results of M. Simon~\cite{Sim02} we can find a
short time solution to the Ricci flow with such an initial metric, using 
smooth approximations to $g$. 
\begin{defn}\label{defn:fair}
	Given a constant $\delta \geqslant 1$, a metric $h$ is 
	$\delta$-fair to $g$ if the curvature of
	$h$ is uniformly bounded, and 
	\[ \frac{1}{\delta} h \leqslant g\leqslant \delta h\,
	\text{ on }M.\]
\end{defn}

M. Simon has shown that there exists a solution to the Ricci flow for a
short time, starting with the singular metric $\hat{g}$, if there exists
a smooth metric $h$ which is $1+\epsilon(n)$-fair to $\hat{g}$ for some
universal constant $\epsilon(n)>0$ only depending on the dimension.
More precisely we have
\begin{thm} Fix
a metric $h$ which is $1+\epsilon(n)$-fair to $g$. Then there exists a
family of metrics $g(t)$ for $t\in(0,T]$ with $T>0$,
which solves the $h$-flow on
this time interval. In addition $h$ is $1+2\epsilon(n)$ fair to $g(t)$ for
$t\in(0,T]$ and $g(t)$ converges to $g$ as $t\to 0$, uniformly on
compact sets. 
\end{thm}

The $h$-flow is the Hamilton-DeTurk flow with background metric $h$.
More precisely $g(t)$ satisfies the equation
\[ \frac{d}{dt} g_{ij} = -2R_{ij} + \nabla_iW_j + \nabla_j W_i,\]
where
\[ W_j = g_{jk}g^{pq}(\Gamma^k_{pq} - \tilde{\Gamma}^k_{pq}).\]
The $\tilde{\Gamma}$ are the Christoffel symbols of the metric $h$. This
flow is equivalent to the Ricci flow \eqref{eq:Ricflow} modulo the
action of diffeomorphisms. The advantage of the $h$-flow is that it is
parabolic, while the unmodified Ricci flow 
\eqref{eq:Ricflow} is only weakly parabolic.  

The way the solution $g(t)$ is constructed is by first taking a sequence of
smoothings $g_\epsilon$ converging to $g$. Then we solve the $h$-flow
for a short time with
initial condition $g_\epsilon$ 
for each small $\epsilon$, obtaining families of
metrics $g_\epsilon(t)$. Then the key point is to show that there exists
a $T>0$ independent of $\epsilon$, such that for 
$t\in(0,T]$ one obtains uniform bounds on all derivatives of
the metrics $g_\epsilon(t)$ of the form
\begin{equation}\label{eq:C2}
	\left|\nabla^k g_\epsilon(t)\right|\leqslant
	\frac{C_k}{t^{k/2}}\,\text{ for }t\in(0,T].
\end{equation}
Here the covariant 
derivative and norm is taken with respect to the metric $h$, and the
constants $C_k$ are independent of $\epsilon$. Then the solution $g(t)$
is extracted as a limit of a subsequence as $\epsilon\to 0$, converging
in $C^k$ for all $k$ on compact sets.

In our application the metric $g$ is asymptotically flat, and so we can
choose $h$ to be flat outside a sufficiently large ball. 
If the metric $g$ is Lipschitz, then the estimates on the derivatives
were improved in M. Simon~\cite{Sim05}. Namely from \cite[Lemma
2.1]{Sim05} we have 
\begin{equation}\label{eq:C1}
	\begin{aligned}
		\left|\nabla g_\epsilon(t)\right| &\leqslant C_1,\\
		|\nabla^2 g_\epsilon(t)| &\leqslant \frac{C_2}{\sqrt{t}},
		\text{ for } t\in(0,T], 
	\end{aligned}
\end{equation}
with $C_1, C_2$ independent of $\epsilon$. Note that here we are
applying Lemma 2.1 from~\cite{Sim05} to the smooth initial metrics
$g_\epsilon$ and if $g$ is Lipschitz, then their smoothings $g_\epsilon$
satisfy a uniform $C^1$ bound, independent of $\epsilon$.

\section{The mass is constant along the Ricci flow}\label{sec:mass}
In this section we will consider a solution of the Ricci flow given by
Theorem~\ref{thm:AFRF}, starting with an asymptotically flat metric
$\hat{g}$ in $C^2_\delta$, 
which has integrable scalar curvature. So we have $g(t)$ for
$t\in[0,T]$ such that $g(0)=\hat{g}$. By choosing $T$ smaller if
necessary, we can
assume that $T\leqslant 1$, and also that outside a ball
$B\subset\mathbf{R}^n$ we have
\[ \frac{1}{2}\delta_{ij} < g_{ij}(t) < 2\delta_{ij} \]
for $t\in[0,T]$. Without loss of generality we can take $B$ to be the
unit ball. In addition we 
have constants $\kappa>0$ and $\delta > (n-2)/2$ such
that on $\mathbf{R}^n\setminus B$ we have
\[ \begin{aligned}
	|g_{ij}(t) - \delta_{ij}| &< \kappa|x|^{-\delta} \\
	|\nabla g_{ij}(t)| &< \kappa|x|^{-\delta-1}\\
	|\nabla^2 g_{ij}(t)| &< \kappa|x|^{-\delta-2}.
\end{aligned}\]
Finally, we will assume that we have a uniform lower bound on the scalar
curvature of $\hat{g}$, and a bound for its $L^1$ norm, i.e. for some $K
> 0$ and the $\kappa$ from above, we have
\[ \begin{gathered} R(\hat{g}) > -K \\
	\int_M |R(\hat{g})| \,dV < \kappa.
\end{gathered}\]
The constants below will depend on $\kappa, K$ and $\delta$. It will be
important that the constant in Lemma~\ref{lem:Rneg} below only depends on $K$. 
The evolution equation of the scalar curvature is (see
Hamilton~\cite{Ham82})
\[ \frac{\partial}{\partial t} R = \Delta R + 2|Ric|^2. \]
Since $R(g(t))$ decays at infinity, we can use
the maximum principle to see that $R(g(t))> -K$ for all $t$.

We first construct some cutoff functions that we will use later. 

\begin{lem}\label{lem:cutoff}
	For any $r_1,r_2$ with $1 < r_1$ and $2r_1 < r_2$,
	there exists a function
	\[ f_{r_1,r_2} : M \to [0,2] \]
	such that
	\[ \begin{gathered}
		f_{r_1,r_2}(x)=1/r_1^2, \text{ if }\, |x| < r_1, \\
		f_{r_1,r_2}(x)\geqslant 1, \text{ if }\, 2r_1 \leqslant
		|x|\leqslant r_2,\\
		f_{r_1,r_2}(x)\leqslant |x|^{-n-1}, \text{ for }\,
		|x|> 2r_2.
	\end{gathered}\]
	In addition there is a constant $C$ independent of $r_1,r_2$ such
	that
	\[ \Delta f_{r_1,r_2} \leqslant Cf_{r_1,r_2},\]
	where $\Delta$ is the Laplacian with respect to any of the
	metrics $g(t)$, $t\in[0,T]$. 
\end{lem}
\begin{proof}
	Let us begin by letting $\phi:[0,\infty)\to [1,2]$ be a
	smooth function such that $\phi(x)=1$ for $x < 1$ and
	$\phi(x)=2$ for $x > 2$. Then define the function
	$g:\mathbf{R}^n \to [1,2]$ by letting $g(x) = \phi(|x|)$. 
	This function satisfies $|\nabla^2g|\leqslant C$ for some $C$.
	Now define
	\[ g_{r_1}(x) = g\left(r_1^{-1}x\right) -1 + \frac{1}{r_1^2}.\]
	Then $|\nabla^2 g_{r_1}|\leqslant r_1^{-2}C$ and
	$g_{r_1}\geqslant r_1^{-2}$ everywhere, so these functions satisfy
	\[ |\nabla^2 g_{r_1}| \leqslant Cg_{r_1}.\]
	
	We will now define another set of functions $h_{r_2}$ on
	$\mathbf{R}^n$. First
	we define $h:\mathbf{R}^n\to [0,1]$ 
	to be such that $h(x)=1$ for $|x|<1$ and
	$h(x)=|x|^{-n-1}$ for $|x|>2$. Then we can check that $|\nabla^2
	h|\leqslant Ch$ for some $C$. Now define
	\[ h_{r_2}(x) = h(r_2^{-1}x),\]
	from which it follows that as long as $r_1 \geqslant 1$, we have
	\[|\nabla^2 h_{r_2}|\leqslant Ch_{r_2}.\]

	Finally, if $r_1 > 1$ and $r_2 > 2r_1$,
	we define $f_{r_1,r_2}:\mathbf{R}^n\to[0,2]$ to be
	\[ f_{r_1,r_2}(x) = g_{r_1}(x) h_{r_2}(x).\]
	Then $f_{r_1,r_2}(x) = g_{r_1}(x)$ for $|x| < r_2$, and
	$f_{r_1,r_2}(x) = (1+r_1^{-2})h_{r_2}(x)$ for $|x| > 2r_1$. It
	follows that the inequality
	\[ |\nabla^2 f_{r_1,r_2}|\leqslant Cf_{r_1,r_2}\]
	holds on $\mathbf{R}^n$. 

	We now simply transfer this function to the asymptotically flat
	manifold using the asymptotic coordinates. This is possible
	since $f_{r_1,r_2}$ is constant in the ball of radius $r_1$. The
	inequality
	\[ \Delta f_{r_1,r_2} \leqslant Cf_{r_1,r_2}\]
	(with a larger choice of $C$)
	follows because in the asymptotic coordinates the metric $g(t)$
	is uniformly equivalent to the Euclidean metric.
\end{proof}

We will also use the following simple ODE result several times.

\begin{lem}\label{lem:ODE}
	Suppose that the function $F(t)$ satisfies
	\[ \frac{d}{dt} F(t) \leqslant AF(t) + B,\]
	for some constants $A,B\geqslant 0$. Then for $t\in[0,1]$ we have
	\[ F(t) \leqslant e^AF(0) + Be^A. \]
\end{lem}

We now study the integrability of the scalar curvature along the Ricci
flow. The following lemma will also be useful when looking at the evolution
of the ADM mass. 
\begin{lem}\label{lem:intRdecay}
	Suppose that
	\[ \int_{M\setminus B_r}|R(\hat{g})|dV<\eta(r)\] 
	with $\eta(r)$ going to zero as $r\rightarrow\infty$. 
	Then there exists a function $\tilde{\eta}(r)$ depending on
	$\eta, \kappa, K$ and $\delta$, but not on $t$, such that
	for $t\in[0,T]$ we have
	\[ \int_{M\setminus B_r}|R(g(t))|dV<\tilde\eta(r),\]
	and $\tilde\eta(r)$ goes to zero as $r\rightarrow\infty$. In
	particular $R(g(t))\in L^1$ for $t\in[0,T]$.  
\end{lem}	
\begin{proof}
	Let us write $\hat{R}$ for $R(\hat{g})$ and $R$ for $R(g(t))$. 
	For any $\epsilon > 0$ define the function
	(following~\cite{Yok08}) 
	\[ u = \sqrt{R^2+\epsilon}.\]
	We can compute that along the Ricci flow
	\[\begin{aligned}
		\frac{\partial u}{\partial t} &= u^{-1}R(\Delta R +
		2|Ric|^2) \\ 
		\Delta u &= -u^{-3}R^2|\nabla R|^2 + u^{-1}|\nabla R|^2 +
		u^{-1}R\Delta R \geqslant u^{-1}R\Delta R,
	\end{aligned}\]
	where we used that $u^{-2}R^2 \leqslant 1$. It follows that
	\begin{equation}\label{eq:dudt}
		\frac{\partial u}{\partial t} \leqslant \Delta u +
		2|Ric|^2.
	\end{equation}
	
	We will now use the cutoff functions $f_{r_1,r_2}$ from
	Lemma~\ref{lem:cutoff}. 
	Since $u$ is bounded, the decay of $f_{r_1,r_2}$ at infinity implies
	that $f_{r_1,r_2}u$ is integrable for all $t$. We can compute
	\[ \frac{d}{dt} \int_M f_{r_1,r_2}u\,dV \leqslant 
		\int_M f_{r_1,r_2}(\Delta u +
		2|Ric|^2 - uR)\,dV. \]
	We have that $R>-K$ for all $t$, and in addition the decay
	condition on the metric implies that outside the ball $B$ we
	have $|Ric|^2 < C_0|x|^{-2\delta - 4} < C_0|x|^{-n-2}$ for some
	constant $C_0$. The integral of $|Ric|^2$ outside $B_{r_1}$ is
	therefore of the order $r_1^{-2}$. 
	It follows that we have
	\begin{equation} \begin{aligned}
		\frac{d}{dt}\int_M f_{r_1,r_2}u\,dV 
		&\leqslant \int_M f_{r_1,r_2}\Delta u +
		Cf_{r_1,r_2}u\,dV + \frac{C_1}{r_1^2}+
		4\int_{M\setminus B_{r_1}}|Ric|^2dV \\
		&\leqslant \int_M u\Delta f_{r_1,r_2} + 
		Cf_{r_1,r_2}u\,dV + C_2r_1^{-2} \\
		&\leqslant 2C\int_M f_{r_1,r_2}u\,dV + 
		C_2r_1^{-2}. 
	\end{aligned}\end{equation}
	It follows using Lemma~\ref{lem:ODE} that
	\[ \int_M f_{r_1,r_2}u\,dV \leqslant C_3\int_M f_{r_1,r_2}
	\sqrt{\hat{R}^2+\epsilon}\,dV + C_4r_1^{-2}. \]
	Letting $\epsilon \to 0$, this implies
	\[ \int_M f_{r_1,r_2} |R|\,dV \leqslant C_3\int_M f_{r_1,r_2} 
	|\hat{R}|\,dV 
	+C_4r_1^{-2}.\]
	Finally, noticing that
	\[\int_{B_{r_2}\setminus B_{2r_1}}|R|dV\leqslant
	\int_M f_{r_1,r_2} |R|\,dV,\]
	we let $r_2\rightarrow\infty$ to get
	\[\begin{aligned}
		\int_{M\setminus B_{2r_1}}|R|dV 
		&\leqslant C_3\int_{B_{r_1}} \frac{1}{r_1^2} |\hat{R}|\,dV
		+2C_3\int_{M\setminus B_{r_1}}|\hat{R}|dV+C_4r_1^{-2}\\
		&\leqslant C\left(r_1^{-2} + \eta(r_1)\right).
	\end{aligned}\]
	This proves the result we want, with $\tilde{\eta}(2r) =
	C(r^{-2} + \eta(r))$. 
\end{proof}

\begin{rem} The fact that the scalar curvature remains integrable for positive time
along the Ricci flow was shown under stronger decay conditions on the
metric by Dai-Ma~\cite{DM07} and Oliynyk-Woolgar~\cite{OW08}. Also, if
we only wanted to show that $R(g(t))\in L^1$ for $t > 0$, then we could
use a simpler argument using the cutoff functions $h_r$ from the proof
of Lemma~\ref{lem:cutoff}. 
\end{rem}

\begin{lem}\label{lem:nablaR}
	For any $t>0$ we have
	\[ \lim_{r\to \infty} \int_{\partial B_r}|\nabla R|\,dS = 0.\]
	Moreover if $t_0 > 0$ then 
	this convergence is uniform for $t\in [t_0, T]$. 
\end{lem}
\begin{proof}
	We first use the local maximum principle for parabolic
	equations (Theorem 7.36 in Lieberman~\cite{Lieb96})
	to obtain pointwise bounds for $R$ depending on the $L^1$
	bound. We could also use the local estimate obtained by Moser
	iteration \cite[Theorem 6.17]{Lieb96}. We will use the equation 
	\[ \frac{\partial}{\partial t} R = \Delta R + 2|Ric|^2.\]
	In addition we work outside the ball $B$, where the asymptotic
	coordinates $x_i$ are defined. 
	As long as $t\in[t_0/2,T]$, the metrics along the flow are all uniformly
	equivalent to the Euclidean metric outside $B$, and also their
	derivatives are controlled. This means that 
	the ellipticity of $\Delta$ and the Christoffel symbols (which
	appear in the first order derivative term in $\Delta R$) 
	are controlled uniformly. Let $\tau\in[t_0,T]$ 
	and $p\in M\setminus B$, and choose $r$ with  
	$0 < r \leqslant\sqrt{\tau-t_0/2}$. 
	Then we apply the local maximum principle to the
	parabolic cylinder 
	\[ Q(r)=\{(x,t)\,:\, |x-p|<r, \tau-r^2 < t < \tau\},\]
	so we obtain
	\[ \sup_{Q(r/2)} |R| \leqslant C_r\left( \int_{Q(r)} |R|\,dV +
	\sup_{Q(r)} |Ric|^2\right),\]
	where $C_r$ is a constant depending on $r$.

	We now apply local $L^p$ estimates and the Sobolev inequalities 
	on the parabolic cylinder
	$Q(r/2)$. Note that on this cylinder as above,
	we control all derivatives of the metric. It follows that we get
	\[ \sup_{Q(r/4)} |\nabla R| \leqslant C'_r\left(\sup_{Q(r/2)}|R| +
	\sup_{Q(r/2)} |Ric|^2\right),  \]
	where again $C'_r$ depends on $r$. In sum we obtain that
	\begin{equation}\label{eq:gradR}
		\sup_{Q(r/4)}|\nabla R| \leqslant C(r) \left(\int_{Q(r)}
		|R|\,dV + \sup_{Q(r)}|Ric|^2\right). 
	\end{equation}

	Now we cover the sphere $\partial B_a$ with balls $B_i$ of radius $r/4$.
	This can be achieved so that each point is covered by at most
	$c(n)$ of the balls $4B_i$, where
	$c(n)$ depends on the dimension. Applying (\ref{eq:gradR}) and
	integrating, we find that
	\[ \int_{\partial B_a} |\nabla R(\tau)|\,dS \leqslant
	C(n,r)\left( \int_{A(a,r)} |R|\,dV + \sup_{A(a,r)}|
	Ric|^2\right),\]
	where $A(a,r)$ is the annulus
	\[ A(a,r) = \{ (x,t)\,:\, d(x,\partial B_a) < r,\, \tau-r^2 < t<
	\tau\}.\]
	The key point is that $C(n,\tau)$ does not depend on $a$ so we can
	let $a\to\infty$. The integrability result 
	Lemma~\ref{lem:intRdecay}, and
	the fact that the metric is asymptotically flat in
	$C^2_\delta$ for some $\delta > (n-2)/2$, now implies the result
	we want. 
\end{proof}

We can now show that the mass is constant along the Ricci flow. This was
shown under stronger decay conditions by Oliynyk-Woolgar~\cite{OW08} and
Dai-Ma~\cite{DM07}.

\begin{cor}\label{cor:mass-constant}
	The ADM mass is preserved under the Ricci flow. 
\end{cor}
\begin{proof}
	Using the integration by parts formula \eqref{eq:mparts} we have
	\begin{equation}\label{eq:mgt}
		m(g(t)) =  \int_{M\setminus B_r} R(g(t))\, dV + \int_{\partial B_r}
		g(t)_{ij,j} - g(t)_{jj,i}\,dS^i + O(r^{-\lambda}),
	\end{equation}
	for some $\lambda>0$, where the constant in $O(r^{-\lambda})$ is
	independent of $t$. Let us write
	\[ m_r(g(t)) = \int_{\partial B_r} g_{ij,j} - g_{jj,i}\,dS^i.\]
	We can use the surface measure $dS^i$ with respect to the fixed
	metric $\hat{g}$ (or even the Euclidean metric), 
	since the metrics $g(t)$ are all asymptotically equal. Then
	\[ \frac{d}{dt} m_r(g(t)) = \int_{\partial B_r} -2R_{ij,j} +
	2R_{,i}\,dS^i.\]
	The derivative $R_{ij,j}$ differs from the covariant derivative
	$R_{ij;j}$ by terms of the form $\Gamma^p_{ij}R_{pj}$, which are
	of order $|x|^{-2\delta-3}$, so integrating over $\partial B_r$
	contributes $O(r^{-\lambda})$. Using this, together with the
	contracted Bianchi identity $2R_{ij;j}= R_{,i}$ we have
	\[ \left|\frac{d}{dt} m_r(g(t))\right| \leqslant \int_{\partial B_r}
	|\nabla R|\,dS + Cr^{-\lambda}.\]
	Now let $\epsilon > 0$. As long as $t\in[t_0,T]$ for some $t_0>0$ 
	we can choose $r\gg 1$ (using Lemma~\ref{lem:nablaR}) such that
	\begin{equation}\label{eq:ddtmr}
		\left|\frac{d}{dt}m_r(g(t))\right| \leqslant \epsilon,
	\end{equation}
	and also, using Lemma~\ref{lem:intRdecay} together with
	Equation~\eqref{eq:mgt}, we can ensure that $|m(g(t))-m_r(g(t))|
	< \epsilon$. These two bounds imply that for $s,t\in[t_0,T]$ we
	have $|m(g(t))-m(g(s))| < (2+T)\epsilon$. Since this is true for
	any $\epsilon > 0$, and we can also choose $t_0>0$ arbitrarily,
	this shows that $m(g(t))$ is constant for $t > 0$. 

	What remains is to show that $\lim_{t\to 0} m(g(t))=m(g(0))$.
	This follows from Equation~\eqref{eq:mgt} applied to $g(0)$ and
	$g(t)$. Indeed, using Lemma~\ref{lem:intRdecay} we have a
	function $\epsilon(r)$ which goes to zero as $r\to \infty$
	(which absorbs the $O(r^{-\lambda})$ as well), such that 
	\[\begin{aligned} |m(g(t)) - m(g(0))| < \epsilon(r) &+ \left| \int_{\partial
		B_r} g(t)_{ij,j} - g(0)_{ij,j}\,dS^i\right| \\
		&+ \left|\int_{\partial B_r} g(t)_{jj,i} - g(0)_{jj,i}\,dS^i\right|.
	\end{aligned}\]
	Letting $t\to 0$ and then $r\to \infty$, we get $\lim_{t\to 0}
	m(g(t)) = m(g(0))$. This completes the proof. 
\end{proof}

We will need the following lemma when showing that the scalar curvature
is non-negative for positive time,
when starting the flow with the singular metric. 
\begin{lem}\label{lem:Rneg}
	For any $t\in[0,T]$ we have
	\[ \int_{\{R(g(t))<0\}} |R(g(t))|\,dV_t \leqslant e^K\int_{\{
	R(\hat{g})<0\}} |R(\hat{g})|\,dV,\]
	where $-K$ is the lower bound for $R(\hat{g})$. 
\end{lem}
\begin{proof}
	Let us simply write $R$ for $R(g(t))$. 
	For $\delta > 0$ define the function 
	\[ v = \sqrt{R^2+\delta} - R,\]
	which is a smoothing of $2\max\{0,-R\}$. From Equation
	(\ref{eq:dudt}) we have that
	\[ \frac{\partial v}{\partial t} \leqslant \Delta v.\]
	Now compute
	\[ \begin{aligned}
		\frac{d}{dt} \int_{B_r} v\,dV &\leqslant \int_{B_r}
		\Delta v - vR\,dV \\
		&\leqslant \int_{\partial B_r} |\nabla v|\,dS +
		K\int_{B_r} v\,dV,
	\end{aligned}\]
	where we used that the lower bound $R(\hat{g}) > -K$ is
	preserved along the flow. 
	At the same time $|\nabla v|\leqslant 2|\nabla R|$, so we have
	\[ \frac{d}{dt} \int_{B_r} v\,dV \leqslant 2\int_{\partial B_r}
	|\nabla R|\,dS + K\int_{B_r} v\,dV. \]

	We know from Lemma~\ref{lem:nablaR}
	that for each $t>0$ the integral of $|\nabla R|$ on
	$\partial B_r$ goes to zero as $r\to\infty$, and 
	that this convergence is uniform as long
	as $t\in[\tau,T]$ for some $\tau > 0$. So for any $\eta,\tau > 0$ we can
	choose $r(\eta)$ such that for $r > r(\eta)$ and $t\in[\tau,T]$
	we have
	\[ 2\int_{\partial B_r} |\nabla R|\,dS < \eta. \]
	It follows then using Lemma~\ref{lem:ODE}
	that for $t\in[\tau,T]$ and $r > r(\eta)$, 
	\[ \int_{B_r} v(t)\,dV_t \leqslant e^K\int_{B_r}
	v(\tau)\,dV_\tau + e^K\eta.\]
	We can now let $\delta\to 0$
	in the definition of $v$, so that $v\to 2\max\{0,-R\}$. It
	follows that
	\[ \int_{B_r\cap \{R(t) < 0\}} |R(t)|\,dV_t \leqslant e^K
	\int_{B_r\cap \{R(\tau) < 0\}} |R(\tau)|\,dV_\tau + e^K\eta. \]
	Now we let $r\to \infty$ and $\eta\to 0$, using the fact
	that $|R|$ is integrable on $M$ by Lemma~\ref{lem:intRdecay}. We get
	\[ \int_{\{R(t) < 0\}} |R(t)|\,dV_t \leqslant e^K\int_{\{R(\tau)
	< 0\}} |R(\tau)|\,dV_\tau.\]
	Since we can do this for any $\tau > 0$, we have
	\[ \int_{\{R(t) < 0\}} |R(t)|\,dV_t \leqslant e^K
	\int_{\{R(\hat{g}) < 0\}}
	|R(\hat{g})|\,dV, \]
	which is what we wanted to prove. 
\end{proof}

\section{The mass of a limit metric}\label{sec:limit}
In this section we will study what we can say about the mass of a metric
$g$, if a sequence of
asymptotically flat metrics $g_\epsilon$ converge to $g$ locally
uniformly. Under convergence in a suitable topology it is known that
$m(g)=\lim m(g_\epsilon)$ (see for example Lee-Parker~\cite[Lemma
9.4]{LP87}). Under weaker assumptions 
we will show that $m(g)\leqslant \liminf m(g_\epsilon)$. 

\begin{thm}\label{thm:mass-limit}
	Suppose that $g_\epsilon\to g$ locally uniformly in $C^2$ and
	that for some $\kappa > 0$ we have the asymptotic decay conditions
	\begin{equation}\label{eq:af} 
		\Vert g_{\epsilon,ij} -
		\delta_{ij}\Vert_{C^{1,\alpha}_\delta(M\setminus K)}
		< \kappa
	\end{equation}
	for some $\delta > (n-2)/2$. In addition we require that
	$R(g_\epsilon)\in L^1$
	for all $\epsilon > 0$, and the scalar curvature of $g_\epsilon$
	is almost non-negative in the following sense:
	\begin{equation}\label{eq:Rnonneg}
		\int_{\{R(g_\epsilon) < 0\}} |R(g_\epsilon)|\,dV < \epsilon.
	\end{equation}
	Then $R(g)\geqslant 0$, $m(g)$ is defined and
	\[ m(g) \leqslant \liminf_{\epsilon\to 0} m(g_\epsilon),\]
	assuming that the limit is finite.
\end{thm}

\begin{lem}\label{lem:Rtint}
	The scalar curvature of $g$ is non-negative and integrable. 
\end{lem}
\begin{proof}
	The fact that $R(g)\geqslant 0$ follows easily by taking the
	limit of Inequality~\eqref{eq:Rnonneg}. 

	Integrating Equation \eqref{eq:R} by parts, and using
	\eqref{eq:expand},  we get that for some $\lambda > 0$
	(depending on $\delta > (n-2)/2$)
	\begin{equation}\label{eq:intparts}
		\begin{aligned}
		\int_{B_r\setminus B_1} R(g)\,dV = &\int_{\partial B_r}
		g_{ij,j} - g_{jj,i}\,dS^i + O(r^{-\lambda}) \\
		&- \int_{\partial B_1} 
		|g|^{1/2}g^{ij}\left(\Gamma_j-\frac{1}{2}\partial_j(\log
		|g|)\right)\,dS^i \\
		&- \int_{B_r\setminus B_1} \frac{1}{2}g^{ij}\Gamma_i
		\partial_j(\log |g|) + g^{ij}g^{kl}g^{pq} \Gamma_{ikp}
		\Gamma_{jql}\, dV.
		\end{aligned}
	\end{equation}
	The constant in $O(r^{-\lambda})$ only depends on $\kappa$ in
	the decay bounds \eqref{eq:af}. Using these decay bounds again,
	and the bound on the
	negative part of $R(g_\epsilon)$ in
	Equation~\eqref{eq:intparts}
	we get
	\[ \int_{B_r\setminus B_1} |R(g_\epsilon)|\,dV \leqslant
	\int_{\partial B_r} (g_\epsilon)_{ij,j} - (g_\epsilon)_{jj,i}
	\,dS^i + C,\]
	where $C$ is independent of $\epsilon$ and $r$. 
	Taking $r\to\infty$ we have
	\[ \int_{M\setminus B_1} |R(g_\epsilon)|\,dV \leqslant
	m(g_\epsilon) + C.\]
	Since $|R(g_\epsilon(t))|\to |R(g(t))|$ pointwise (and the
	volume forms also converge pointwise), we find that $R(g(t))$ is
	integrable, using Fatou's Lemma, as long as $\liminf
	m(g_\epsilon)$ is finite.
\end{proof}

\begin{proof}[Proof of Theorem~\ref{thm:mass-limit}]
	By Lemma \ref{lem:Rtint} and the decay conditions \eqref{eq:af}
	we know that the mass of $g(t)$ is defined. 
	From the integration by parts formula \eqref{eq:mparts} we have
	\[ m(g) =  \int_{M\setminus B_r} R(g)\, dV + \int_{\partial B_r}
		g_{ij,j} - g_{jj,i}\,dS^i + O(r^{-\lambda}), \]
	Using the same formula again with
	$g_\epsilon$ instead of $g$, we get
	\begin{equation}\label{eq:mg}
	\begin{aligned}
		m(g) =& m(g_\epsilon) + O(r^{-\lambda}) \\
			&+ \int_{M\setminus B_r} R(g)\,dV -
			\int_{M\setminus B_r} R(g_\epsilon)\,dV_\epsilon
			\\
			&+ \int_{\partial B_r} g_{ij,j} - g_{jj,i}\,dS^i
			- \int_{\partial B_r} (g_\epsilon)_{ij,j} -
			(g_\epsilon)_{jj,i}\,dS_\epsilon^i. 
	\end{aligned}\end{equation}

	Note that for a fixed $r > 0$, the metrics $g_\epsilon$
	converge uniformly in $C^2$ to $g$ on $\partial B_r$ 
	as $\epsilon\to 0$. At the same time, using \eqref{eq:Rnonneg}
	together with Fatou's Lemma we have
	\[ \int_{M\setminus B_r} R(g)\,dV \leqslant \liminf_{\epsilon\to
	0} \int_{M\setminus B_r} R(g_\epsilon)\,dV.\]
	Therefore taking the limit as $\epsilon\to 0$
	in \eqref{eq:mg} for a fixed $r$, we
	get
	\[ m(g)\leqslant \liminf_{\epsilon\to 0} m(g_\epsilon) +
	O(r^{-\lambda}). \]
	Now taking $r\to\infty$ we get the required result. 
\end{proof}

\section{Proof of the main theorem}\label{sec:main}
We now let $\hat{g}$ be a metric with corners across a hypersurface,
following Miao~\cite{Miao02}. For the definition see the Introduction.
We assume that $\hat{g}$ is asymptotically flat in $C^2_\delta$ for some
$\delta > (n-2)/2$, but note that $\hat{g}$ only has to be $C^2$ outside
a compact set. 
The important result for us is that the smoothing
procedure in~\cite[Proposition 3.1]{Miao02}
gives us metrics $\hat{g}_\epsilon$ which satisfy
the following conditions, with a fixed $K> 0$: 
\begin{equation}\begin{gathered}\label{eq:approx}
	\hat{g}_\epsilon = \hat{g}\,\text{ outside }B, \\
	(1-\epsilon)\hat{g}\leqslant
	\hat{g}_\epsilon\leqslant(1+\epsilon)\hat{g},\\
	R(\hat{g}_\epsilon) > -K,\\
	\int_{\{R(\hat{g}_\epsilon)<0\}} |R(\hat{g}_\epsilon)|\,dV_\epsilon <
	\epsilon.
\end{gathered}\end{equation}
The last
condition holds since the metrics $\hat{g}_\epsilon$ can be constructed such
that the set $\{R(\hat{g}_\epsilon)<0\}$ has measure less than 
$K^{-1}\epsilon$ and at the same time $R(\hat{g}_\epsilon)>-K$ everywhere
for some uniform constant $K$. 

Let $g(t)$ be the solution of the Hamilton-DeTurk flow with
initial metric $\hat{g}$, and background metric $h=\hat{g}_\epsilon$ for
sufficiently small $\epsilon$, constructed by Simon~\cite{Sim02}. In
fact we can modify $h$ to be equal to the Euclidean metric outside a
compact set, and it will still be $1+\epsilon(n)$-fair to $\hat{g}$. Recall
that the $h$-flow with initial metric $\hat{g}$ is obtained by taking
the limit of the $h$-flows with initial metrics $\hat{g}_\epsilon$. Let
us write $g_\epsilon(t)$ for the $h$-flow with initial metric
$\hat{g}_\epsilon$. 

We first want to show
that the metrics $g(t)$ are asymptotically flat for $t >
0$. We are not able to show that $g(t)$ is asymptotically flat in
$C^2_\delta$, but only in $C^{1,\alpha}_{\delta'}$ with $\alpha > 0$
small, and $\delta' = \delta-\alpha$. 
This is enough for the positive
mass theorem to hold (see e.g. Lee-Parker~\cite{LP87}). 

\begin{lem}\label{lem:af}
	There is a $T > 0$ independent of $\epsilon$, and for 
	any $t_0\in(0,T]$ 
	there is a constant $\kappa > 0$ such that for $t\in [t_0,T]$
	and $\epsilon > 0$
	we have
	\begin{equation}\label{eq:decay2}
		\Vert g_\epsilon(t) - h\Vert_{C^{1,
		\alpha}_{\delta-\alpha}} \leqslant\kappa.
	\end{equation}
	By taking $\epsilon\to 0$ the same estimate holds for $g(t)$. In
	particular if $\delta > (n-2)/2$ and $\alpha$ is small enough,
	then $\delta'=\delta-\alpha$ satisfies $\delta' 
	> (n-2)/2$, and $g(t)$ is
	asymptotically flat in $C^{1,\alpha}_{\delta'}$. 
\end{lem}
This is proved using standard techniques, using the maximum principle.
We give the proof in the Appendix. 

Next we check that the other conditions of Theorem~\ref{thm:mass-limit}
are satisfied. 

\begin{lem}\label{thm:positive}
	For a fixed $t > 0$, the metrics $g_\epsilon(t)$ satisfy the
	hypotheses of Theorem~\ref{thm:mass-limit}. 
\end{lem}
\begin{proof}
	Each $\hat{g}_\epsilon$ is asymptotically flat, so
	we can apply the results of Section~\ref{sec:mass} to the Ricci
	flow starting at $\hat{g}_\epsilon$. Lemma~\ref{lem:intRdecay}
	implies that $g_\epsilon(t)$ has integrable scalar curvature.  
	In addition we can apply
	Lemma~\ref{lem:Rneg}, and we find that
	\[ \int_{\{ R(g_\epsilon(t))<0\}} |R(g_\epsilon(t))|\,dV_t
	\leqslant e^K \int_{ \{R(\hat{g}_\epsilon)<0\}}
	|R(\hat{g}_\epsilon)|\,dV_\epsilon < e^K\epsilon.\]
\end{proof}

Finally, we give the proofs of Theorems \ref{thm:Miao} and
\ref{thm:zeromass}. 

\begin{thm}\label{thm:diffeo}
	The $C^0$ metric $\hat{g}$ has $m(\hat{g})\geqslant 0$. If
	$m(\hat{g})=0$, then $\hat{g}$ is the Euclidean metric up to a
	$C^{1,\alpha}$ change of coordinates. 
\end{thm}
\begin{proof}
	Consider the solution $g(t)$ to the $h$-flow as above. Note that
	$m(\hat{g}_\epsilon)=m(\hat{g})$ for all $\epsilon$ since we are
	only changing the metric outside a ball. For each $\epsilon >
	0$, there are diffeomorphisms $\phi_{\epsilon,t}$ such that 
	$\phi^*_{\epsilon,t}g_\epsilon(t)$ is 
	a solution of the Ricci flow, and $\phi_{\epsilon,0}$ is the
	identity. Applying the results of
	Section~\ref{sec:mass}, Corollary~\ref{cor:mass-constant}
	implies that
	$m(\phi^*_{\epsilon,t} g_\epsilon(t)) = m(\hat{g}_\epsilon) = 
	m(\hat{g})$ 
	for $t>0$. Since under our decay conditions the mass is
	independent of the choice of asymptotic coordinates (see
	Bartnik~\cite[Theorem 4.2]{Bar86}), we also have
	$m(g_\epsilon(t)) = m(\hat{g})$. 
	Finally, we use that $g(t) = \lim g_\epsilon(t)$, and
	Theorem~\ref{thm:mass-limit} implies that $m(g(t))\leqslant
	m(\hat{g})$. 

	By Theorem \ref{thm:mass-limit} the metrics
	$g(t)$ have non-negative scalar curvature for $t>0$ and they are
	asymptotically flat in $C^{1,\alpha}_{\delta'}$ for some
	$\delta' > (n-2)/2$, so 
	by the positive mass theorem, $m(g(t))\geqslant
	0$. We therefore have $m(\hat{g})\geqslant 0$. 

	Suppose now that $m(\hat{g})=0$. Then necessarily $m(g(t))=0$ for
	$t>0$. From the equality case of the positive mass theorem,
	$g(t)$ is flat for $t>0$. This means that the $h$-flow is just
	acting by diffeomorphisms, according to the equations
	\begin{equation}\label{eq:h-flow1}
		\begin{aligned}
		\frac{\partial}{\partial t} g_{ij} &= \nabla_i W_j +
		\nabla_j W_i \\
		W_j &= g_{jk}g^{pq} (\Gamma_{pq}^k -
		\tilde{\Gamma}_{pq}^k),
	\end{aligned}
	\end{equation}
	where $\tilde{\Gamma}$ are the Christoffel symbols of $h$. 
	From the ODE
	\begin{equation}\label{eq:h-flow2} \begin{aligned}
		\frac{\partial}{\partial t} (\phi_t(p)) &= W(\phi_t(p),t);
		\\
		\phi_T(p) &= \mathrm{Id}(p),
	\end{aligned}
	\end{equation}
	we obtain a family of diffeomorphisms $\phi_t$ for $t>0$ such that 
	$g(t) = \phi_t^* g(T)$.  We can think of $\phi_t$ as a
	perturbation of the identity 
	given by $\phi_t(p)=\mathrm{Id}(p)+\psi_t(p)$.
	Since the metric $g$ we started with was Lipschitz, the bound
	\eqref{eq:C1} applies, so 
	$|\nabla g(t)| < C_1$ for $t>0$. Since the
	$g(t)$ are bounded in $C^1$, the $\psi_t$ are bounded in $C^2$.
	This can be seen by the method of Taylor~\cite{Tay06}, in
	particular Equation (2.13), which expresses the second
	derivatives of $\psi_t$ in terms of its first derivatives and
	the Christoffel symbols of $g(t)$ and $g(T)$. The formula is
	\[ \frac{\partial^2 \phi_t^m}{\partial x^i\partial x^j} = 
	\Gamma^k_{ij}\frac{\partial \phi_t^m}{\partial x^k} -
	\hat{\Gamma}^m_{kl} \frac{\partial\phi_t^k}{\partial x^i}
	\frac{\partial \phi_t^l}{\partial x^j},\]
	where $\phi_t^m$ are the components of $\phi_t$, and
	$\hat{\Gamma}$ are the Christoffel symbols of $g(T)$. The same
	formula holds for the second derivatives of $\psi_t$ (in terms
	of the first derivatives of $\phi_t$). 
	Moreover we can bound $\psi_t$ in
	$C^0$, using Equation~\eqref{eq:h-flow2} together with the fact
	that the vector field $W$ is bounded in $C^0$ from its
	definition~\eqref{eq:h-flow1}. So the $\psi_t$ are bounded in
	$C^2$, and 
	we can extract a subsequence converging in $C^{1,\alpha}$ to
	some $\psi_0$. It follows that $g = \phi_0^* g(T)$, where
	$\phi_0 = \mathrm{Id} + \psi_0$, and $g(T)$
	is the flat Euclidean metric on $\mathbf{R}^n$. So up to a
	$C^{1,\alpha}$ change of coordinates, $g$ is the flat metric. 
\end{proof}

\section{Appendix - Proof of Lemma~\ref{lem:af}}
In this section we give the proof of Lemma~\ref{lem:af}. Recall that we
have a smooth approximation $\hat{g}_\epsilon$ to our singular metric
$\hat{g}$, and $g_\epsilon(t)$ is the solution of the $h$-flow with
initial metric $\hat{g}_\epsilon$. The key point in this lemma is that
since $\hat{g}$ is Lipschitz, the smoothings satisfy a 
uniform weighted $C^1$-bound $\Vert \hat{g}_\epsilon -
h\Vert_{C^1_\delta} < C$, where $C$ is independent of $\epsilon$. We
will show that this bound (with a larger $C$) is preserved for positive
time along the flow, and in addition if $t \geqslant t_0 > 0$, then 
$|\nabla^2 g_\epsilon(t)| < C'\rho^{-\delta-1}$. Note that we do not
obtain the natural decay of $\rho^{-\delta-2}$ for the second
derivatives
(see the remark after the following proof). 
\begin{proof}
	In this proof we will simply write $g$ instead of
	$g_{\epsilon}(t)$. Let 
	\[\eta_{ij} = g_{ij} - h_{ij},\]
	where $h$
	is our background reference metric. From \cite{Sim02,Sim05} we
	know that $g$ is uniformly equivalent to $h$ for $t\in[0,T]$ and
	also we
	have the estimate \eqref{eq:C1}, so for some
	constant $c_0$ we have $|\nabla\eta| < c_0$ for $t\in[0,T]$.
	Let us first show that $|\eta|$ has the required decay. The
	evolution equation for $\eta$ is given by (see \cite[Equation
	(1.5)]{Sim02})
	\begin{equation}\label{eq:ddteta}
		\begin{aligned}
		\frac{\partial}{\partial t} \eta_{ab} =&
		\Delta \eta_{ab} -
		g^{cd}g_{ap}h^{pq}\tilde{R}_{bcqd} -
		g^{cd}g_{bp}h^{pq}\tilde{R}_{acqd} \\
		&+ \frac{1}{2} g^{cd}g^{pq}\big(\nabla_a \eta_{pc}\nabla_b
		\eta_{qd} + 2\nabla_c\eta_{ap}\nabla_{q}\eta_{bd} \\
		&- 2\nabla_c
		\eta_{ap}\nabla_d\eta_{bq}-4\nabla_a\eta_{pc}
		\nabla_d\eta_{bq}\big),
		\end{aligned}
	\end{equation}
	where all the derivatives are with respect to $h$, $\tilde{R}$
	is the curvature of $h$ and $\Delta$ is the operator
	$g^{cd}\nabla_c\nabla_d$. From this we obtain the following
	schematic equation. 
	\[ \frac{\partial}{\partial t} |\eta|^2 = \Delta |\eta|^2 +
	\eta\star \tilde{R} + \eta\star\nabla\eta\star\nabla\eta -
	2h^{ik}h^{jl}g^{pq}\nabla_p\eta_{ij}\nabla_q\eta_{kl},\]
	where $\star$ means an algebraic operation involving
	contractions with respect to $h$ and $g$ and the norms are
	computed with respect to $h$. Since $g$ is
	$\epsilon(n)$-fair to $g$ (see Definition~\ref{defn:fair}) for
	some small $\epsilon(n)$, we have
	\[ h^{ik}h^{jl}g^{pq}\nabla_p\eta_{ij}\nabla_q\eta_{kl}
	\geqslant \frac{3}{4} |\nabla\eta|^2,\]
	where the norm is measured using $h$. In addition we can bound
	\[ |\eta\star\nabla\eta\star\nabla\eta| \leqslant C_1|\eta|^2 +
	\frac{1}{2}|\nabla\eta|^2, \]
	using that $|\nabla\eta|$ is uniformly bounded. It follows that
	\begin{equation}\label{eta2}
		\frac{\partial}{\partial t} |\eta|^2 \leqslant \Delta|\eta|^2
		+ C_1|\eta|^2 + C_2|\tilde{R}|^2 - |\nabla\eta|^2.
	\end{equation}
	We can assume that $h$ is asymptotically flat in $C^k_\delta$
	for any $k$, so in particular
	$|\tilde{R}| < C\rho^{-\delta-2}$. We obtain
	\[ \frac{\partial}{\partial t} |\eta|^2\leqslant \Delta|\eta|^2
	+ C_1|\eta|^2 + C_3\rho^{-2\delta-4}. \]
	Using the formulas
	\[ \begin{aligned}
		\Delta (\rho^{2\delta}|\eta|^2) &= \rho^{2\delta}\Delta |
		\eta|^2 +
		4\delta\rho^{2\delta-1}\nabla\rho\cdot\nabla|\eta|^2 +
		\Delta(\rho^{2\delta})|\eta|^2 \\
		\rho^{-1}\nabla\rho\cdot\nabla(\rho^{2\delta}
		|\eta|^2) &=
		2\delta\rho^{2\delta-2}|\nabla\rho|^2|\eta| +
		\rho^{2\delta-1}\nabla\rho\cdot\nabla|\eta|^2,
	\end{aligned}\]
	together with bounds on $\rho,\nabla\rho,\Delta\rho$, we
	get
	\[ \frac{\partial}{\partial t} \rho^{2\delta}|\eta|^2
	\leqslant \Delta(\rho^{2\delta}|\eta|^2) -
	4\delta\rho^{-1}\nabla\rho\cdot\nabla(\rho^{2\delta}|\eta|^2)
	+ C_4\rho^{2\delta}|\eta|^2 + C_3. \]
	Since $\rho^{2\delta}|\eta|^2 < \kappa_0$ at $t=0$, the
	maximum principle (the version on non-compact manifolds due to
	Ecker-Huisken~\cite{EH91} for example)
	implies that for some $\kappa$ we have
	$\rho^{2\delta}|\eta|^2 < \kappa$ for $t\in[0,T]$. 

	Let us now consider $|\nabla\eta|^2$. From \eqref{eq:ddteta} we
	obtain
	\[\begin{aligned}
		\frac{\partial}{\partial t} |\nabla\eta|^2 =& \Delta |
		\nabla\eta|^2 + \nabla\eta\star\nabla\eta\star
		\nabla^2\eta + \nabla\eta\star\nabla\eta\star
		\nabla\eta\star\nabla\eta \\
		&+ \nabla\eta\star\nabla\eta \star \tilde{R} +
		\nabla\eta\star \nabla\tilde{R} \\
		&-2g^{cd}h^{ik}h^{jl}h^{pq}\nabla_c\nabla_p\eta_{ij}
		\nabla_d\nabla_q\eta_{kl},
	\end{aligned}\]
	where several terms come from commuting derivatives, and also
	from differentiating $g^{ab}$ in $\Delta = g^{ab}\nabla_a
	\nabla_b$. 
	With similar arguments as above, we obtain
	\begin{equation}\label{eq:nablaeta}
		\frac{\partial}{\partial t} |\nabla\eta|^2 \leqslant
		\Delta|\nabla\eta|^2 + C_1|\nabla\eta|^2 +
		C_2\rho^{-2\delta-4} - |\nabla^2\eta|^2.
	\end{equation}
	We will use the negative squared term below, but for now 
	repeating the arguments above, and using that at $t=0$ we have
	$\rho^{2\delta+2}|\nabla\eta|^2<\kappa_0$, we again find that
	for some $\kappa > 0$ we have $\rho^{2\delta+2}|\nabla\eta|^2
	< \kappa$ for $t\in[0,T]$. 

	Finally, consider $|\nabla^2\eta|^2$. Once again from
	\eqref{eq:ddteta}, using similar arguments, we get
	\[ \begin{aligned}
		\frac{\partial}{\partial t} |\nabla^2\eta|^2 =& 
		\Delta |\nabla^2\eta|^2
		+\nabla^2\eta\star\nabla\eta\star\nabla^3\eta +
		\nabla^2\eta\star\nabla^2\eta\star\nabla^2\eta\\
		&+ \nabla^2\eta\star \nabla^2\eta\star \nabla\eta
		\star\nabla\eta + \nabla^2\eta\star \nabla^2\eta\star
		\tilde{R}
		+ \nabla^2\eta\star\nabla\eta\star \nabla\tilde{R} \\
		&+ \nabla^2\eta\star\nabla^2\tilde{R} -
		2g^{pq}h^{ik}h^{jl}h^{ab}h^{cd}\nabla_p\nabla_a\nabla_c\eta_{ij}
		\nabla_q\nabla_b\nabla_d h_{kl}.
	\end{aligned}\]
	From this we obtain
	\[ \frac{\partial}{\partial t} |\nabla^2\eta|^2 \leqslant \Delta
	|\nabla^2\eta|^2 + C_1|\nabla^2\eta|^3 + C_2|\nabla^2\eta|^2 +
	C_3\rho^{-2\delta-2},\]
	where we also used our previous estimate on $|\nabla\eta|$. 
	Now computing, using \eqref{eq:nablaeta}, we find that if we let
	$f = (|\nabla\eta|^2+1)|\nabla^2\eta|^2$ then
	\[ \frac{\partial}{\partial t}f \leqslant \Delta f +
	C_4|\nabla^2\eta|^2 +
	C_5\rho^{-2\delta-2}.\] 
	Now using \eqref{eq:nablaeta} and the bound we already have for
	$|\nabla\eta|$ we get
	\[ \frac{\partial}{\partial t} (C_6|\nabla\eta|^2 + tf)\leqslant
	\Delta(C_6|\nabla\eta|^2 + tf) + C_7\rho^{-2\delta-2}.\]
	It follows that if we write $F = C_6|\nabla\eta|^2 + tf -
	Ct\rho^{-2\delta-2}$ for some constant $C$, then 
	\[ \frac{\partial}{\partial t} F \leqslant \Delta F +
	Ct\Delta(\rho^{-2\delta-2}) - C\rho^{-2\delta-2} +
	C_7\rho^{-2\delta-2}.\]
	Now just as in the proof of Lemma~\ref{lem:cutoff}, there is a
	constant $C_8$ such that
	$\Delta(\rho^{-2\delta-2}) < C_8\rho^{-2\delta-2}$, using that 
	the metrics along the flow are asymptotically
	flat in $C^1_\delta$, by our estimates on $\eta$ and
	$\nabla\eta$. So we have
	\[ \frac{\partial}{\partial t} F \leqslant \Delta F +
	C(C_8t-1)\rho^{-2\delta-2} + C_7\rho^{-2\delta-2}.\]
	As long as $t < (2C_8)^{-1}$, we can choose $C$ sufficiently
	large, so that 
	\[ \frac{\partial}{\partial t} F \leqslant \Delta F. \]
	Using the maximum principle as before, 
	we obtain that for $t\geqslant t_0 >
	0$ and as long as $t < (2C_8)^{-1}$, 
	there is some constant $C_9$, for which $f <
	C_9\rho^{-2\delta-2}$. This implies that
	$|\nabla^2\eta | < C_{10}\rho^{-\delta-1}$
	for some $C_{10}$. Since $|\nabla\eta|$ satisfies the same bound, we
	get the required weighted $C^{1,\alpha}$ bound on $\eta$. 
\end{proof}

\begin{rem}\label{rem:decay}
	It is natural to expect that for the second derivatives
	of $\eta$ we can obtain the improved decay $|\nabla^2\eta| <
	C\rho^{-\delta-2}$ for $t \geqslant t_0 > 0$. This does not seem
	to be the
	case however, since even for the linear heat equation on
	$\mathbf{R}$ the analogous statement does not hold.
	One can check that if $f(x,t)$ is the bounded solution of
	$\partial_t f = \partial^2_x f$ on $\mathbf{R}$, with
	\[ f(x,0) = \frac{1}{1+x^2} \sin(x),\]
	then for each fixed $t_0 \geqslant 0$ and $k$ we have
	$|\partial_x^k f(x,t_0)| = O(|x|^{-2})$, and the decay cannot be
	improved even for positive time. This is why we work in the
	weighted $C^{1,\alpha}$ spaces with sufficiently small $\alpha >
	0$. 
\end{rem}

\end{document}